\documentclass[reqno,a4paper]{amsart}

\usepackage[latin1,utf8]{inputenc}
\usepackage[english]{babel}
\usepackage{eucal,amsfonts,amssymb,amsmath,amsthm,epsfig,mathrsfs}
\usepackage{cancel,soul}
\usepackage{color}
\usepackage{amscd,amsxtra}
\usepackage{enumerate}
\usepackage{latexsym}
\usepackage{bm}
\usepackage{setspace}
\usepackage{textcomp}

\allowdisplaybreaks

\newcounter{ipotesi}

\Alph{ipotesi}
 \makeatletter \@addtoreset{equation}{section}

\makeatother \makeatletter

\newtheorem{thm}{Theorem}[section]
\newtheorem{hyp}[thm]{Hypotheses}{\rm}
{\rm}

\newtheorem{prop}[thm]{Proposition}
\newtheorem{defi}[thm]{Definition}
\newtheorem{rmk}[thm]{Remark}{\rm}

\newcounter{parentenv}

\newcommand{\R}{{\mathbb R}}

\newcommand{\E}{{\mathbb E}}
\newcommand{\N}{{\mathbb N}}

\newcommand{\X}{{\mathcal{X}}}
\newcommand{\K}{{\mathcal{K}}}
\newcommand{\D}{{\mathcal{D}}}
\newcommand{\J}{{\mathcal{D}}}

\newcommand{\ra}{\rightarrow}

\newcommand{\ol}[1]{\overline{#1}}

\newcommand{\OO}{{\mathcal{O}}}
\newcommand{\CO}{\overline{\mathcal{O}}}

\newcommand{\Dom}{{\operatorname{Dom}}}

\newcommand{\set}[1]{{\left\{#1\right\}}}
\newcommand{\pa}[1]{{\left(#1\right)}}

\newcommand{\abs}[1]{{\left|#1\right|}}
\newcommand{\norm}[1]{{\left\|#1\right\|}}

\newcommand{\mi}[1]{{\lbrace #1(t,x)\rbrace_{t\geq 0}}}

\newcommand{\eqsys}[1]{{\left\{\begin{array}{ll}#1\end{array}\right.}}
\newcommand{\tc}{\, \middle |\,}

\makeatletter
\@namedef{subjclassname@2020}{%
  \textup{2020} Mathematics Subject Classification}
\makeatother

\begin{document}

\frenchspacing

\title[Schauder estimates for equations driven by colored noise]{Schauder estimates for stationary and evolution equations associated to stochastic reaction-diffusion equations driven by colored noise}

\author[D.A. Bignamini and S. Ferrari ]{Davide A. Bignamini, Simone Ferrari$^*$}

\address{D.A.B.: Dipartimento di Matematica, Universit\`a degli Studi di Pavia, Via Adolfo Ferrata 5, 27100 PAVIA, Italy}
\email{\textcolor[rgb]{0.00,0.00,0.84}{davideaugusto.bignamini@unipv.it}}

\address{S.F.:  Dipartimento di Matematica e Fisica ``Ennio De Giorgi'', Universit\`a del Salento, Via per Arnesano, 73100 LECCE, Italy}
\email{\textcolor[rgb]{0.00,0.00,0.84}{simone.ferrari@unisalento.it}}

\thanks{$^*$Corresponding author.}



\date{\today}

\begin{abstract}
We consider stochastic reaction-diffusion equations with colored noise and prove Schauder type estimates, which will depend on the color of the noise, for the stationary and evolution problems associated with the corresponding transition semigroup, defined on the Banach space of bounded and uniformly continuous functions.
\end{abstract}

\maketitle

\section{Introduction}
The theory of Schauder regularity estimates for equations driven by differential operators with bounded coefficients was developed throughout the 20th century, see for example \cite{KCL75,KCL80,LSU68,LU68}.
In the late 1990s, a new interest in Schauder regularity estimates for stationary and evolution equations driven by differential operators with unbounded coefficients began to develop, see for instance \cite{CDP96-2, DA6, DL95, Lun97, Pri09}. Besides their obvious analytic interest one of the main motivation for this research is the relationship between second-order elliptic operators and stochastic differential equations, for example in problems such as uniqueness in law, pathwise uniqueness and uniqueness of the martingale problem for stochastic partial differential equations, see for instance \cite{DPF2010,DFPR13,Zam20}.


This paper is devoted to the study of Schauder regularity estimates for the stationary and evolution equations driven by the second order differential operator associated to a stochastic reaction-diffusion equation in a separable Banach space. Let $\mathcal{O}$ be a bounded set of $\R^d$, with $d=1,2,3$. Let $L^2(\mathcal{O})$ be the space of square integrable functions with respect to the $d$-dimensional Lebesgue measure on $\mathcal{O}$, with the usual quotient with respect to the equality almost everywhere, and let $C(\CO)$ be the space of continuous functions on $\ol{\mathcal{O}}$ endowed with the uniform norm. Let $A:\Dom(A)\subseteq C(\CO)\ra C(\CO)$ be a realization of the Laplacian operator with Dirichlet or Neumann boundary conditions. Let $F:C(\CO)\ra C(\CO)$ be a Nemytskii operator given by
\[
F(x)(\xi):=b(\xi,x(\xi)),\qquad \xi\in\ol{\mathcal{O}},\ x\in C(\CO),
\]
where $b:\CO\times \R\ra\R$ is a smooth enough function having a polynomial growth with respect to the second variable (see Hypotheses \ref{RDS1}\eqref{bbbb}). Let $\{W(t)\}_{t\geq 0}$ be a $L^2(\mathcal{O})$-cylindrical Wiener process with respect to a normal filtered probability space $(\Omega,\mathcal{F},\{\mathcal{F}_t\}_{t\geq 0},\mathbb{P})$ (see \cite[Section 4.1.2]{DA-ZA4}, for a definition).
We consider the following stochastic reaction-diffusion equation
\begin{gather}\label{eqF02}
\eqsys{
dX(t,x)=\big[AX(t,x)+F(X(t,x))\big]dt+(-A)^{-\gamma/2}dW(t), & t>0;\\
X(0,x)=x\in C(\CO),
}
\end{gather} 
where $\gamma\in [0,1]$. For any $x\in E:=\overline{\Dom(A)}$ (the closure is taken in $C(\ol{\mathcal{O}})$, with respect the uniform norm), under suitable hypotheses (see \cite[Proposition 6.2.2]{CER1}), equation \eqref{eqF02} has a unique $E$-valued pathwise continuous mild solution $\mi{X}$ (see Definition \ref{Mild1}) and we can define the transition semigroup $\{P(t)\}_{t\geq 0}$ given by
\begin{align}\label{transition}
P(t)\varphi(x):=\E[\varphi(X(t,x))]:=\int_\Omega\varphi(X(t,x))(\omega)\mathbb{P}(d\omega),\qquad \varphi\in {\rm BUC}(E),\ x\in E,\ t\geq 0,
\end{align}
where ${\rm BUC}(E)$ is the space of real valued, bounded and uniformly continuous function on $E$. Using the same arguments as in  \cite[Proposition 3.3]{CDP12} it is possible to prove that $\{P(t)\}_{t\geq 0}$ is a weakly continuous semigroup on ${\rm BUC}(E)$, so we let $N:\Dom(N)\subseteq \rm{BUC}(E)\ra \rm{BUC}(E)$ be its weak generator, namely the unique closed operator such that
\begin{align}\label{uintro}
u:=R(\lambda,N)\varphi=\int^{+\infty}_0e^{-\lambda t}P(t)\varphi dt,\qquad \lambda>0,\ \varphi\in \rm{BUC}(E).
\end{align}
In the case $d=1$, $\CO=[0,1]$ and $\gamma=0$, the authors of \cite{CDP12} prove maximal Schauder regularity estimates for the function $u$ given by \eqref{uintro}, namely if $\varphi$ is $\alpha$-H\"older continuous for some $\alpha\in (0,1)$ then $u$ is twice Fr\'echet differentiable with bounded and continuous derivatives and with second order $\alpha$-H\"older continuous derivative. Even in the case $F\equiv 0$, if $\gamma>0$ in \eqref{eqF02} then it does not seem reasonable to obtain the Schauder estimates similar to the one in \cite{CDP12} for $u$, see for instance \cite[Section 5.1]{LR21}. 

The main purpose of this paper is to show how it is possible to obtain Schauder type estimates for the function $u$ depending on the constant $\gamma\in [0,1]$. Clearly in the case $\gamma=0$ and $d=1$ we will obtain the same result of \cite{CDP12}. Moreover 
we prove Schauder regularity estimates for the mild solution of the following evolution equation 
\begin{align}\label{evol_prob}
\eqsys{\frac{d}{dt}v(t,x)=Nv(t,x)+g(t,x), & t\in(0,T],\ T>0,\ x\in E;\\
v(0,x)=f(x), & x\in E,}
\end{align}
where $f,g$ belong to suitable H\"older spaces (see Section \ref{proofevol}). These Schauder results will be similar to the one in \cite{KCL75,KCL80}, where they were proved for evolution equations driven by a second order operator with bounded coefficients in $\R^d$. 

We stress that the coloring of the noise in \eqref{eqF02} (meaning adding the operator $(-A)^{-\gamma/2}$ in front of $dW(t)$), is not an arbitrary choice. 
Indeed, if $d=2,3$, this choice is necessary to guarantee the existence and uniqueness of a pathwise continuous mild solution for \eqref{eqF02}. 
More precisely 
there exists a positive constant $\gamma_{\mathcal{O}}$, depending on $\mathcal{O}$, such that if $\gamma>\gamma_{\mathcal{O}}$, then \eqref{eqF02} has a unique pathwise continuous mild solution (see \cite[Remark 6.1.1]{CER1} and \cite[Section 4.2]{FU-OR1}).

The results of this paper are based on the estimates contained in \cite[Section 6.5]{CER1} and on the techniques presented in \cite{CL21,LR21} where analogous results were obtained for stationary and evolution equations associated to linear stochastic equations. However, for the semilinear case, we were unable to find results in the literature that link the color of the noise driving \eqref{eqF02} with the H\"older regularity of the solutions of the stationary and evolution equation associated to \eqref{eqF02}, so the results of this paper are the first of their kind for stochastic partial differential equations of this type. Moreover we believe that the results contained in this paper may be the starting point for studying the existence and uniqueness of martingale solutions of stochastic partial differential equations of the type
\begin{gather*}
\eqsys{
dX(t,x)=\big[AX(t,x)+F(X(t,x))+G(X(t,x))\big]dt+(-A)^{-\gamma/2}dW(t), & t>0;\\
X(0,x)=x,
}
\end{gather*}
where $G$ is a H\"older continuous function (we refer to \cite{AD-MA-PR1}, for the case $F\equiv0$). However, this study goes far beyond the purposes of this paper.

Moreover, to give a more complete view of the results in the literature, we wish to point out that in \cite{BF_Schauder,CL19} Schauder regularity estimates along suitable directions are studied, in \cite{LR21} the case of pseudo-differential operators are considered; in \cite{CDP96,PRI1,PZ00} the Gross Laplacian and some of its perturbations are considered; and in \cite{CL21} the non-autonomous linear case is investigated. For other related results see also \cite{ABGP06,ABP05, DA5}.

Before rigorously stating the hypotheses and the main results of this paper, it is necessary to recall some standard definitions and fix the notations.

{\small
\subsection*{Notations}
Let $(\Omega,\mathcal{F},\{\mathcal{F}_t\}_{t\geq 0},\mathbb{P})$ be a filtered probability space. We say that $\{\mathcal{F}_t\}_{t\geq 0}$ is a normal filtration if
\[\mathcal{F}_t=\bigcap_{s>t}\mathcal{F}_s \qquad t\geq 0;\]
and $\mathcal{F}_0$ contains all the elements $A\in\mathcal{F}$ such that $\mathbb{P}(A)=0$. 

Let $\K$ be a Banach space endowed with the norm $\norm{\cdot}_\K$. We denote by $\mathcal{B}(\K)$ the Borel $\sigma$-algebra associated to the norm topology in $\K$. For any $p\geq 1$ we denote by $L^p((\Omega,\mathcal{F},\mathbb{P});\K)$ the space of equivalence classes of measurable function $f:(\Omega,\mathcal{F})\ra (\K,\mathcal{B}(\K))$, with respect to the equality $\mathbb{P}$-almost everywhere equivalent relation, such that
\[
\|f\|^p_{L^p((\Omega,\mathcal{F},\mathbb{P});\K)}:=\E[\norm{f}^p_\K]=\int_\Omega \norm{f(\omega)}^p_\K\mathbb{P}(d\omega)<+\infty,
\]
where the integral is meant in the Bochner sense, see \cite{DU77}.

Let $\K_1$ and $\K_2$ be two real Banach spaces equipped with the norms $\norm{\cdot}_{\K_1}$ and $\norm{\cdot}_{\K_2}$, respectively. For $k\in\N$ we denote by $\mathcal{L}^{(k)}(\K_1;\K_2)$ the space of continuous multilinear maps from $\K_1^k$ to $\K_2$, if $k=1$ we simply write $\mathcal{L}(\K_1;\K_2)$, while if $\K_1=\K_2$ we write $\mathcal{L}^{(k)}(\K_1)$. 

We denote by $B_b(\K_1;\K_2)$ the set of the bounded and Borel measurable functions from $\K_1$ to $\K_2$. If $\K_2=\R$, then we simply write $B_b(\K_1)$. We denote by ${\rm BUC}(\K_1;\K_2)$ the space of bounded and uniformly continuous functions from $\K_1$ to $\K_2$. We consider ${\rm BUC}(\K_1;\K_2)$ with the norm
\[
\norm{f}_{\infty}=\sup_{x\in \K_1}\|f(x)\|_{\K_2}.
\]
If $\K_2=\R$ we simply write  ${\rm BUC}(\K_1)$. 

For any $\alpha\in (0,1)$, we say that a bounded function $f:\K_1\ra\K_2$ is $\alpha$-H\"older continuous (Lipschitz continuous, respectively) if
\begin{align*}
[f]_{C_b^\alpha(\mathcal{K}_1;\mathcal{K}_2)}&:=\sup_{\substack{x,y\in \K_1\\ x\neq y}}\frac{\norm{f(x)-f(y)}_{\K_2}}{\norm{x-y}^{\alpha}_{\K_1}}<+\infty;\\ 
\Bigg([f]_{{\rm Lip}_b(\mathcal{K}_1;\mathcal{K}_2)}&:=\sup_{\substack{x,y\in \K_1\\ x\neq y}}\frac{\norm{f(x)-f(y)}_{\K_2}}{\norm{x-y}_{\K_1}}<+\infty,\text{ respectively}\Bigg).
\end{align*}
We denote by $C_b^\alpha(\K_1;\K_2)$ (${\rm Lip}_b(\K_1;\K_2)$, respectively) the subspace of ${\rm BUC}(\K_1;\K_2)$ of the $\alpha$-H\"older (Lipschitz continuous, respectively) functions. The spaces $C_b^\alpha(\K_1;\K_2)$ and ${\rm Lip}_b(\K_1;\K_2)$ are a Banach spaces if they are endowed with the norms
\begin{align*}
\norm{f}_{C_b^\alpha(\K_1;\K_2)}:=\norm{f}_\infty+[f]_{C_b^\alpha(\mathcal{K}_1;\mathcal{K}_2)};\qquad \norm{f}_{{\rm Lip}_b(\K_1;\K_2)}:=\norm{f}_\infty+[f]_{{\rm Lip}_b(\mathcal{K}_1;\mathcal{K}_2)}.
\end{align*}  
If $\K_2=\R$ we simply write $C_b^\alpha(\K_1)$ and ${\rm Lip}_b(\K_1)$.

We say that a function $f:\K_1\ra \K_2$ is Fr\'echet differentiable at the point $x\in \K_1$, if there exists $L_x\in\mathcal{L}(\K_1;\K_2)$ such that 
\begin{align*}
\lim_{\norm{h}_{\K_1}\ra 0}\frac{\|f(x+h)-f(x)-L_xh\|_{\K_2}}{\norm{h}_{\K_1}}=0.
\end{align*}
When it exists, the operator $L_x$ is unique and it is called Fr\'echet derivative of $f$ at the point $x\in \K_1$ and we denote it by $\J f(x):=L_x$. We say that $f$ is twice Fr\'echet differentiable at the point $x\in \K_1$ if the map $\J f:\K_1\ra \mathcal{L}(\K_1;\K_2)$ is Fr\'echet differentiable at the point $x\in \K_1$. 
We call second order Fr\'echet derivative of $f$ at the point $x\in \K_1$ the unique continuous bilinear form $\J^2 f(x):\K_1\times \K_1\ra \K_2$ defined by
\[
\J^2 f(x)(h,k):=\J (\J f(x)h)k,\quad h,k\in \K_1.
\]
For any $k\in\N$, in an analogous way, we can define the notion of $k$ times Fr\'echet differentiable function $f$ and we denote by $\J^k f(x)$ its $k$ tikmes Fr\'echet derivative operator at the point $x\in \K_1$, in particular $\J^k f(x)$ belongs to $\mathcal{L}^{(k)}(\K_1;\K_2)$. In a similar way we define the Gateaux derivative operator and we denote it by $\D_G$. For further standard results concerning differentiability in infinite dimension we refer to \cite[Chapter 7]{FHH11}.

For $k\in\N$, we denote by ${\rm BUC}^{k}(\K_1)$ the space of bounded, uniformly continuous and $k$ times Fr\'echet differentiable functions $f:\K_1\ra\R$ such that $\D^i_Rf\in {\rm BUC}(\K_1;\mathcal{L}^{(i)}(\K_1;\R))$. We endow ${\rm BUC}^{k}(\K_1)$  with the norm
\[
\norm{f}_{{\rm BUC}^{k}(\K_1)}:=\norm{f}_{\infty}+\sum^k_{i=1}\sup_{x\in\K_1}\|\J^i_Rf(x)\|_{\mathcal{L}^{(i)}(\K_1;\R)}.
\]
For $k\in\N$ and $\alpha\in (0,1)$, we denote by $C_b^{k+\alpha}(\K_1)$ the subspace of ${\rm BUC}^{k}(\K_1)$ of functions $f:\K_1\ra\R$ such that $\D^k_Rf\in C_b^{\alpha}(\K_1;\mathcal{L}^{(k)}(\K_1;\R))$. We endow $C_b^{k+\alpha}(\K_1)$ with the norm
\[
\norm{f}_{C_b^{k+\alpha}(\K_1)}:= \norm{f}_{{\rm BUC}^{k}(\K_1)}+[\J^k f]_{{C^\alpha_b(\mathcal{K_1};\mathcal{L}^{(k)}(\K_1;\R)})}.
\]
For $\beta>0$ and $\beta\notin\N$ let $[\beta]$ and $\{\beta\}$ be the integer and the fractional part of $\beta$, respectively. We denote by $C_b^{\beta}(\K_1)$ the space $C_b^{[\beta]+\{\beta\}}(\K_1)$.

We need to recall the definition of Zygmund space. The Zygmund space $\mathcal{Z}^1(\K_1;\K_2)$ is the subspace of ${\rm BUC}(\K_1;\K_2)$ consisting of functions $f$ such that
\begin{align*}
[f]_{\mathcal{Z}^1(\K_1;\K_2)}:=\sup_{\substack{x\in \K_1\\ h\in \K_1\setminus\set{0}}}\frac{\|f(x+2h)-2f(x+h)+f(x)\|_{\K_2}}{\norm{h}_{\K_1}},
\end{align*}
is finite. The space $\mathcal{Z}^1(\K_1;\K_2)$ is a Banach space if endowed with the norm
\[\norm{f}_{\mathcal{Z}^1(\K_1;\K_2)}:=\norm{f}_\infty+[f]_{\mathcal{Z}^1(\K_1;\K_2)}.\]
If $\mathcal{K}_2=\R$, then we simply write $\mathcal{Z}^1(\K_1)$. We recall that every Lipschitz function belongs to $\mathcal{Z}^1(\K_1;\K_2)$, but, even in the case $\K_1=\K_2=\R$, there are bounded and continuous functions not belonging to $\mathcal{Z}^1(\K_1;\K_2)$ (see, for example, \cite{Tri95}). If $k\geq 2$, then we consider the following Zygmund spaces $\mathcal{Z}^k(\K_1)$ defined as
\begin{align*}
\mathcal{Z}^k(\K_1):=\{f\in{\rm BUC}^{k-1}(\K_1)\,|\,\J^{k-1}f\in\mathcal{Z}^1(\K_1;\mathcal{L}^{k-1}(\K_1;\R))\},
\end{align*}
endowed with the norm
\begin{align*}
\|f\|_{\mathcal{Z}^k(\K_1)}:=\|f\|_{{\rm BUC}^{k-1}(\K_1)}+[\J^{k-1}f]_{\mathcal{Z}^1(\K_1;\mathcal{L}^{k-1}(\K_1;\R))}.
\end{align*}
}

\subsection*{Hypotheses and main results}
In this paper we will work in the same framework of \cite[Chapter 6]{CER1} which we recall in the following hypotheses.
\begin{hyp}\label{RDS1}
Let $\OO$ be a bounded subset of $\R^d$, with $d=1,2,3$ and let $(\Omega, \mathcal{F},\{\mathcal{F}_t\}_{t\geq 0},\mathbb{P})$ be a normal filtered probability space.
\begin{enumerate}[\rm (a)]
\item $\{W(t)\}_{t\geq 0}$ is a $L^2(\mathcal{O})$-cylindrical Wiener process defined on $(\Omega, \mathcal{F},\{\mathcal{F}_t\}_{t\geq 0},\mathbb{P})$.

\item $A:\Dom(A)\subseteq C(\CO)\ra C(\CO)$ is a realization of the Laplacian operator with Dirichlet or Neumann boundary conditions. Throughout the paper we let $E=\ol{{\rm Dom}(A)}$, where the closure is taken in $C(\ol{\mathcal{O}})$ with respect to the uniform norm.

\item\label{gamma} Let $\gamma\in [0,1]$ be such that for any $p\geq 1$ and $T>0$ the map $\omega\mapsto W_A(\cdot)(\omega)$
belongs to $L^p((\Omega, \mathcal{F},\mathbb{P});C([0,T];E))$, where
\begin{align*}
W_A(t):=\int^t_0e^{(t-s)A}(-A)^{-\gamma/2}dW(s),\qquad t\geq 0.
\end{align*}

\item\label{bbbb} $F:C(\CO)\ra C(\CO)$ is the Nemytskii operator defined as
\[
F(x)(\xi):=b(\xi,x(\xi)),\qquad \xi\in\CO,\ x\in C(\CO),
\]
where $b:\ol{\mathcal{O}}\times \R\ra\R$ is such that, for every $\xi\in\ol{\mathcal{O}}$, the map $z\mapsto b(\xi,z)$ belongs to $C^3(\R)$ and the maps $\D_z^jb:\ol{\mathcal{O}}\times \R\ra\R$ are continuous, for $j=0,1,2,3$. Moreover there exists an interger $m\geq 0$ such that
\begin{align*}
\sup_{\xi\in\ol{\mathcal{O}}}\sup_{z\in \R}\frac{\abs{\J^j_zb(\xi,z)}}{1+|z|^{\max(0,2m+1-j)}}<+\infty,\qquad j=0,1,2,3.
\end{align*}
Furthermore if $m\geq 1$, then there exist $a>0$, $c\in\R$ and $\theta\geq 0$ such that
\begin{align}\label{superdis}
\sup_{\xi\in\ol{\mathcal{O}}}(b(\xi,z+h)-b(\xi,z))h\leq -a h^{2(m+1)}+c(1+\abs{z}^\theta),\qquad z,h\in \R.
\end{align}
\end{enumerate}
\end{hyp}

Hypotheses \ref{RDS1} are required to obtain existence, uniqueness and some regularity properties for the mild solution of \eqref{eqF02} as we will show in Section \ref{prel}. In Section \ref{RMKS} we present some significant examples of $A$, $F$, $\OO$ and $\gamma$ which satisfy Hypotheses \ref{RDS1}.

The following is one of the main results of the paper.

\begin{thm}\label{THM_reg_BUC}
Assume Hypotheses \ref{RDS1} hold true. Let $\lambda>0$, $f\in{\rm BUC}(E)$ and let $u=R(\lambda,N)f$ be the function introduced in \eqref{uintro}.
\begin{enumerate}[\rm (i)]
\item\label{THM_reg_BUC1} If $\gamma=0$, then $u\in \mathcal{Z}^2(E)$ and there exists $C=C(\lambda)>0$ such that $\|u\|_{\mathcal{Z}^2(E)}\leq C\|f\|_\infty$.

\item\label{THM_reg_BUC2} If $\gamma\in(0,1)$, then $u$ belongs to $C_b^{1+(1-\gamma)/(1+\gamma)}(E)$ and there exists $C=C(\gamma,\lambda)>0$ such that
\begin{align*}
\|u\|_{C_b^{1+(1-\gamma)/(1+\gamma)}(E)}\leq C\|f\|_\infty.
\end{align*}

\item\label{THM_reg_BUC3} If $\gamma=1$, then $u\in\mathcal{Z}^1(E)$ and there exists $C=C(\lambda)>0$ such that $\|u\|_{\mathcal{Z}^1(E)}\leq C\|f\|_\infty$.
\end{enumerate}
\end{thm}

Theorem \ref{THM_reg_BUC} explains how the color of the noise influences the regularity of the function $u$ defined in \eqref{uintro}. In the second main result of this paper we will show what happens if the function $f$ in \eqref{uintro} is bounded and $\alpha$-H\"older continuous.

\begin{thm}\label{Thm_Schauder_Rd}
Assume Hypotheses \ref{RDS1} hold true. Let $\lambda>0$, $\alpha\in (0,1)$, $f\in C_b^{\alpha}(E)$ and let $u=R(\lambda,N)f$ be the function introduced in \eqref{uintro}.
\begin{enumerate}[\rm (i)]
\item\label{Thm_Schauder_Rd1} If $\gamma\in (0,1]$ and $0<\alpha<\frac{2\gamma}{1+\gamma}$ then $u$ belongs to $C_b^{1+\alpha+(1-\gamma)/(1+\gamma)}(E)$ and there exists $C=C(\gamma,\lambda,\alpha)>0$ such that
\begin{align*}
\|u\|_{C_b^{1+\alpha+(1-\gamma)/(1+\gamma)}(E)}\leq C\|f\|_{C_b^\alpha(E)}.
\end{align*}

\item\label{Thm_Schauder_Rd2} If $\gamma\in(0,1)$ and $\alpha=\frac{2\gamma}{1+\gamma}$, then $u$ belongs to $\mathcal{Z}^1(E)$  there exists $C=C(\gamma,\lambda)>0$ such that
\begin{align*}
\|u\|_{\mathcal{Z}^1(E)}\leq C\|f\|_{C_b^{2\gamma/(1+\gamma)}(E)}.
\end{align*}

\item\label{Thm_Schauder_Rd3} If $\gamma\in[0,1)$ and $\alpha>\frac{2\gamma}{1+\gamma}$ then $u$ belongs to $C_b^{2+\alpha-2\gamma/(1+\gamma)}(E)$ and there exists $C=C(\gamma,\lambda,\alpha)>0$ such that
\begin{align*}
\|u\|_{C_b^{2+\alpha-2\gamma/(1+\gamma)}(E)}\leq C\|f\|_{C_b^\alpha(E)}.
\end{align*}
\end{enumerate}
\end{thm}
In the case $\gamma=0$ and $\OO=[0,1]$, Theorem \ref{Thm_Schauder_Rd} is the result contained in \cite{CDP12} instead Theorem \ref{THM_reg_BUC} is an extension of the results in \cite{CDP12} to the case $\alpha=0$.  In the other cases Theorems  \ref{THM_reg_BUC} and \ref{Thm_Schauder_Rd} are new results for stationary equations associated to stochastic partial differential equations similar to \eqref{eqF02}. 

We have decided to write Theorems \ref{THM_reg_BUC} and \ref{Thm_Schauder_Rd} as you see them so as to separate the integer and fractional part in the exponent of $C_b^\beta(E)$. In this way it is easy to see the number of derivatives that the solution $u$ possesses and the H\"olderianity exponent of the highest order derivative. The following two results are a rewriting of Theorems \ref{THM_reg_BUC} and \ref{Thm_Schauder_Rd} in such a way that the gain in regularity that the solution $u$ undergoes becomes more evident.

\par\addvspace{7pt}
\noindent {\bf Theorem \ref{THM_reg_BUC}'.}\textit{ Assume Hypotheses \ref{RDS1} hold true. Let $\lambda>0$, $f\in{\rm BUC}(E)$ and let $u=R(\lambda,N)f$ be the function introduced in \eqref{uintro}.}
\begin{enumerate}[\rm (i)]
\item \textit{If $2/(1+\gamma)\notin \N$, then $u\in C_b^{2/(1+\gamma)}(E)$ and there exists $C=C(\lambda,\gamma)>0$ such that $\|u\|_{C_b^{2/(1+\gamma)}(E)}\leq C\|f\|_\infty$.
}
\item \textit{If $2/(1+\gamma)\in \N$, then $u\in\mathcal{Z}^{2/(1+\gamma)}(E)$ and there exists $C=C(\lambda,\gamma)>0$ such that $\|u\|_{\mathcal{Z}^{2/(1+\gamma)}(E)}\leq C\|f\|_\infty$.
}
\end{enumerate}
\par\addvspace{7pt}
\noindent The following is a rewriting of Theorem \ref{Thm_Schauder_Rd}.

\par\addvspace{7pt}
\noindent {\bf Theorem \ref{Thm_Schauder_Rd}'.}\textit{ Assume Hypotheses \ref{RDS1} hold true. Let $\lambda>0$, $\alpha\in (0,1)$, $f\in C_b^{\alpha}(E)$ and let $u=R(\lambda,N)f$ be the function introduced in \eqref{uintro}.}
\begin{enumerate}[\rm (i)]
\item \textit{If $\alpha+2/(1+\gamma)\notin \N$, then $u$ belongs to $C_b^{\alpha+2/(1+\gamma)}(E)$ and there exists $C=C(\lambda,\gamma,\alpha)>0$ such that $\|u\|_{C_b^{\alpha+2/(1+\gamma)}(E)}\leq C\|f\|_{C_b^\alpha(E)}$.
}
\item \textit{If $\alpha+2/(1+\gamma)\in \N$, then $u$ belongs to $\mathcal{Z}^{\alpha+2/(1+\gamma)}(E)$ and there exists $C=C(\lambda,\gamma,\alpha)>0$ such that $\|u\|_{\mathcal{Z}^{\alpha+2/(1+\gamma)}(E)}\leq C\|f\|_{C_b^\alpha(E)}$.
}
\end{enumerate}
\par\addvspace{7pt}
Thus it follows by Theorems \ref{THM_reg_BUC}' and \ref{Thm_Schauder_Rd}' that the solution $u$ always gains a regularity of $\alpha+2/(1+\gamma)$, that is, its regularity depends on both the smoothness of the initial datum $f$ and the coloring of the noise. We will also prove spatial Schauder regularity estimates (similar to those in Theorems \ref{THM_reg_BUC} and \ref{Thm_Schauder_Rd}) for the mild solution of the evolution equation \ref{evol_prob} (Theorems \ref{THM_evol_reg_BUC} and \ref{Thm_evol_Schauder_Rd}).

We conclude this introduction with some observations. At the cost of some technical modifications in the proofs, we belive to be possible to prove the same results contained in this paper in the more abstract setting considered in \cite{AD-BA-MA1,Big21,Mas08}. But this goes beyond the scope of the present paper. Moreover, if $m=1$ in Hypothesis \ref{RDS1}\eqref{bbbb}, then $F\in {\rm BUC}^3(L^2(\OO);L^2(\OO))$, hence all the results contained in this paper holds true by replacing $E$ with $L^2(\OO)$. 

We would like to point out that in \cite{CL19,CL21,LR21} the authors obtain results similar to Theorems \ref{THM_reg_BUC} and \ref{Thm_Schauder_Rd} for linear equations by replacing the space of uniformly continuous and bounded functions with those of continuous and bounded functions. For technical reasons, Schauder regularity results involving semilinear equations in infinite dimension are not known to be extendable to continuous and bounded functions.

This paper is organized as follows. In Section \ref{prel} we recall some results about the mild solution of the \eqref{eqF02} and the transition semigroup \eqref{transition} that are fundamental in the proofs of this paper. In Section \ref{proofstaz} we prove Theorems  \ref{THM_reg_BUC} and \ref{Thm_Schauder_Rd}. In Section \ref{proofevol} we prove spatial Schauder regularity estimates for the mild solution of the evolution equation \eqref{evol_prob}. Finally in Section \ref{RMKS} we present some examples and we suggest some possible generalization of the results of this paper.

\section{Preliminaries results}\label{prel}

In this section we recall some preliminaries results that are fundamental in the proofs of the main results of this paper.

\subsection{The transition semigroup}
We recall that a $L^2(\mathcal{O})$-cylindrical Wiener process $\{W(t)\}_{t\geq 0}$ is defined by
\begin{align}\label{cil}
W(t):=\sum_{k=1}^{+\infty}\beta_k(t)e_k
\end{align}
where $\{e_k\}_{k\in\N}$ is an orthonormal basis of $L^2(\mathcal{O})$ and $\{\beta_1(t)\}_{t\geq 0},\ldots,\{\beta_k(t)\}_{t\geq 0},\ldots$ are independent real Brownian motions. The convergence of the series in \eqref{cil} is meant in the space  $L^2((\Omega,\mathcal{F},\mathbb{P});C([0,T];\K))$, where $\K$ is an appropriate separable Hilbert space such that  $L^2(\mathcal{O})$ is continuously embedded in $\K$ (we refer to \cite{AP-RI1}, \cite[Section 4.1.2]{DA-ZA4}) and \cite[Chapter 2]{LI-RO1} for an overview about cylindrical Wienere processes).

\begin{defi}\label{Mild1}
Assume Hypotheses \ref{RDS1} hold true. For any  $x\in E$ we call mild solution of \eqref{eqF02} any $E$-valued process $\mi{X}$ such that, for any $t\geq 0$, it holds
\begin{align*}
 X(t,x)=e^{tA}x+\int_0^te^{(t-s)A}F(X(s,x))ds+W_A(t),\qquad\mathbb{P}\text{-a.e.},
\end{align*}
where $\{W_A(t)\}_{t\geq 0}$ is the stochastic convolution process defined in Hypothesis \ref{RDS1}\eqref{gamma}.
\end{defi}

We summarize the results contained in  \cite[Proposition 6.2.2 and Theorem 6.2.3]{CER1} in the following proposition.

\begin{prop}
Assume that Hypotheses \ref{RDS1} hold true. For any $x\in E$ the stochastic partial differential equation \eqref{eqF02} has a unique mild solution such that the map $\omega\mapsto X(\cdot,x)(\omega)$ belongs in $L^p((\Omega,\mathcal{F},\mathbb{P}); C([0,T];E))$, for any $T>0$ and $p\geq 1$. Moreover for any $t>0$
\begin{align*}
\sup_{x\in E}\|X(t,x)\|_E\leq k(t)t^{-1/2m},\qquad\mathbb{P}\text{-a.e.},
\end{align*} 
where the random variable $k(t)$ is defined by
\begin{align}\label{stimaunif}
k(t)=c(1+\norm{W_A(t)}^{\max(\theta/(1+2m),1)}_E),\qquad\mathbb{P}\text{-a.e.},
\end{align}
and $c$, $\theta$, $m$ are the constants appearing in Hypothesis \ref{RDS1}\eqref{bbbb}.
Furthermore there exists a constant $\eta\in \R$ such that for any $t>0$ and $x,y\in E$ it holds
\begin{align}\label{liplip}
\norm{X(t,x)-X(t,y)}_E\leq e^{t\eta}\norm{x-y}_E,\qquad\mathbb{P}\text{-a.e.}
\end{align}
\end{prop}
\noindent We remark that $m\geq 1$ in the Hypothesis \ref{RDS1}\eqref{bbbb} then  \eqref{superdis} is needed to obtain \eqref{stimaunif}.

Using the mild solution $\mi{X}$ of \eqref{eqF02}, we define the family of bounded and linear operator $\{P(t)\}_{t\geq 0}$ on ${\rm BUC}(E)$ defined by
\[
P(t)\varphi(x):=\E[\varphi(X(t,x))]:=\int_\Omega\varphi(X(t,x))(\omega)\mathbb{P}(d\omega),\quad \varphi\in {\rm BUC}(E),\;x\in E,\;t\geq 0.
\]
In the same way of \cite[Proposition 3.3]{CDP12} it is possible to prove that $\{P(t)\}_{t\geq 0}$ is a weakly continuous semigroup on ${\rm BUC}(E)$ and define its weak generator. To this aim we need to recall the definition of $\mathscr{K}$-convergence as introduced in \cite{Cer94} (see also \cite[Appendix B]{CER1}).
\begin{defi}
A sequence $(\varphi_n)_{n\in\N}\subseteq {\rm BUC}(E)$ is $\mathscr{K}$-convergent to $\varphi\in{\rm BUC}(E)$ if
\begin{align}\label{K-conv}
\sup_{n\in\N}\|\varphi_n\|_\infty<+\infty,\qquad \lim_{n\ra+\infty}\sup_{x\in K}|\varphi_n(x)-\varphi(x)|=0;
\end{align}
for any compact subset $K$ of $E$. If \eqref{K-conv} holds true we write $\mathscr{K}\text{-}\lim_{n\ra +\infty}\varphi_n=\varphi$.
Moreover if $I\subseteq\R$ and $t_0$ is an accumulation point for $I$, we say that the family $\{\varphi_t\}_{t\in I}\subseteq {\rm BUC}(E)$ is $\mathscr{K}$-convergent to $\varphi\in{\rm BUC}(E)$ as $t$ approches $t_0$ and we write $\mathscr{K}\text{-}\lim_{t\ra t_0}\varphi_t=\varphi$,
if the sequence $(\varphi_{t_n})_{n\in\N}$ is $\mathscr{K}$-convergent to $\varphi$ as $n$ tends to infinity, for any sequence $(t_n)_{n\in\N}\subseteq I$ converging to $t_0$. 
\end{defi}

We are now able to give the definition of weakly continuous semigroup.

\begin{defi}
A semigroup of bounded linear operators $\{P(t)\}_{t\geq 0}$ on ${\rm BUC}(E)$ is said to be weakly continuous if
\begin{enumerate}
\item The family $\{P(t)\varphi\, |\, t\in [0,T]\} \subseteq {\rm BUC}(E)$ is equi-uniformly continuous, for any $\varphi\in {\rm BUC}(E)$ and $T>0$;

\item there exist $M>0$ and $w\in\R$ such that for every $t\geq 0$ it holds $\|P(t)\|_{\mathcal{L}({\rm BUC}(E))}\leq Me^{wt}$;

\item for every $\varphi\in {\rm BUC}(E)$ it holds $\mathscr{K}\text{-}\lim_{t\ra 0^+}P(t)\varphi=\varphi$.

\item for every $t\geq 0$, $\varphi\in {\rm BUC}(E)$ and sequence  $(\varphi_{n})_{n\in\N}\subseteq {\rm BUC}(E)$ which is $\mathscr{K}$-convergent to $\varphi$, it holds $\mathscr{K}\text{-}\lim_{n\ra+\infty}P(t)\varphi_n=P(t)\varphi$, where the limit is uniform with respect to $t\in[0,T]$ for any $T>0$.
\end{enumerate}
\end{defi}

The weak generator of $\{P(t)\}_{t\geq 0}$ (in the sense of \cite[Definition B.1.5]{CER1}) is the unique close operator $N:\Dom(N)\subseteq \rm{BUC}(E)\ra \rm{BUC}(E)$ such that
\[
R(\lambda,N)\varphi=\int^{+\infty}_0e^{-\lambda t}P(t)\varphi dt,\qquad \lambda>0,\ \varphi\in \rm{BUC}(E).
\]


\subsection{Regularity of the transition semigroup}

Exploiting \eqref{liplip} and the integral form of the transition semigroup $\{P(t)\}_{t\geq 0}$ we obtain the following result.

\begin{prop}\label{daviderompeilcazzo}
Assume that Hypotheses \ref{RDS1} hold true and let $\alpha\in (0,1)$. If $t\geq 0$, then
\[
\norm{P(t)}_{\mathcal{L}({\rm BUC}(E))}\leq 1,\qquad \norm{P(t)}_{\mathcal{L}(C_b^{\alpha}(E))}\leq e^{\alpha\eta t},\qquad \norm{P(t)}_{\mathcal{L}({\rm Lip}_b(E))}\leq e^{\eta t};
\] 
where $\eta\in\R$ is the constant appearing in \eqref{liplip}.
\end{prop}

In the next proposition we summarize the results in \cite[Proposition 6.4.1 and Theorems 6.5.1]{CER1}.

\begin{prop}
Assume that Hypotheses \ref{RDS1} hold true. 
\begin{enumerate}[\rm (1)]
\item For any $t>0$ the map $E\ni x\mapsto X(t,x)$ belongs to $L^p((\Omega,\mathcal{F},\mathbb{P}); (E,\mathcal{B}(E)))$ and it is three times Gateaux differentiable. Moreover the following estimates hold true
\begin{align}\label{stimemild}
\sup_{x\in E}\int^t_0\|A^{\gamma/2}\D^j_GX(s,x)(h_1,\ldots,h_j)\|^2_{L^2(\mathcal{O})}ds\leq c_j(t)t^{1-\gamma}\prod^j_{i=1}\norm{h_i}^2_E,\qquad \mathbb{P}\text{-a.e.}
\end{align}
where $j=1,2,3$ and $h_1,\ldots,h_j\in E$. Here $\gamma\in [0,1]$ is the constant appearing in Hypothesis \ref{RDS1}\eqref{gamma} and $\{c_j(t)\}_{t\geq 0}$ are increasing processes having finite moments of any order.

\item If $t>0$ and $\varphi\in \rm{BUC}(E)$, then $P(t)\varphi\in {\rm BUC}^3(E)$, and for any $x,h\in E$
\begin{align}\label{BEL}
\D P(t)\varphi(x)h=\frac{1}{t}\E\left[\varphi(X(t,x))\int^t_0\langle A^{\gamma/2}\D^j_GX(s,x)h,dW(s)\rangle_{L^2(\mathcal{O})}\right].
\end{align}
Moreover for $j=0,1,2,3$ there exists a constant $C_j>0$ such that for any $t>0$, $x\in E$ and $\varphi\in{\rm BUC}(E)$ it holds
\begin{align}\label{Goku}
\|\J^jP(t)\varphi(x)\|_{\mathcal{L}^{(j)}(E;\R)}\leq C_j (\min\{1,t\})^{-j(1+\gamma)/2}\norm{\varphi}_{{\rm BUC}(E)}.
\end{align}
\end{enumerate}
\end{prop}

Combining \eqref{stimemild} and \eqref{BEL} we obtain the following result.

\begin{prop}
Assume that Hypotheses \ref{RDS1} hold true. For $j=1,2,3$ there exists a constant $K_j>0$ such that for any $t>0$, $x\in E$ and $\varphi\in {\rm BUC}^1(E)$
\begin{align}\label{GokuC}
\|\J^jP(t)\varphi(x)\|_{\mathcal{L}^{(j)}(E;\R)}\leq K_j (\min\{1,t\})^{-(j-1)(1+\gamma)/2}\norm{\varphi}_{{\rm BUC}^1(E)}.
\end{align}
\end{prop}

Exploiting Proposition \ref{daviderompeilcazzo} and proceeding exactly as in \cite[Proposition 3.5 and (3.10)]{CDP12} we get the following result. 

\begin{prop}
Assume that Hypotheses \ref{RDS1} hold true. For $j=1,2,3$ there exists a constant $K'_j>0$ such that for any $t>0$, $x\in E$ and $\varphi\in {\rm Lip}_b(E)$
\begin{align}\label{Gokulip}
\|\J^jP(t)\varphi(x)\|_{\mathcal{L}^{(j)}(E;\R)}\leq K'_j (\min\{1,t\})^{-(j-1)(1+\gamma)/2}\norm{\varphi}_{{\rm Lip}_b(E)}.
\end{align}
\end{prop}

We refer to \cite[Appendix A]{CDP12} for a proof of the following interpolation result.

\begin{thm}\label{Clip}
Assume Hypotheses \ref{RDS1} hold true and let $\vartheta\in (0,1)$. Up to an equivalent renorming, it holds $({\rm BUC}(E),{\rm Lip}_b(E))_{\vartheta,\infty}=C_b^\vartheta(E)$.
\end{thm}

We recall a classical interpolation result that we will use in the paper (see \cite[Theorem 1.12 of Chapter 5]{CR88} for a proof).

\begin{thm}\label{classic}
Let $\K_0,\K_1,\mathcal{H}_0,\mathcal{H}_1$ be Banach spaces such that $\mathcal{H}_0\subseteq \mathcal{K}_0$ and $\mathcal{H}_1\subseteq \mathcal{K}_1$ with continuous embeddings. If $T$ is a linear mapping such that $T:\K_0\ra\K_1$ and $T:\mathcal{H}_0\ra\mathcal{H}_1$ and for some $N_{\K},N_{\mathcal{H}}>0$ it hold
\begin{align*}
\|Tx\|_{\K_1}&\leq N_{\K}\|x\|_{\K_0},\qquad x\in \K_0;\\
\|Ty\|_{\mathcal{H}_1}&\leq N_{\mathcal{H}}\|y\|_{\mathcal{H}_0},\qquad y\in \mathcal{H}_0,
\end{align*}
then, for every $\vartheta\in(0,1)$, $T$ maps $(\K_0,\mathcal{H}_0)_{\vartheta,\infty}$ in $(\K_1,\mathcal{H}_1)_{\vartheta,\infty}$ and
\begin{align*}
\|Tx\|_{(\K_1,\mathcal{H}_1)_{\vartheta,\infty}}&\leq N_{\K}^{1-\vartheta}N_{\mathcal{H}}^{\vartheta}\|x\|_{(\K_0,\mathcal{H}_0)_{\vartheta,\infty}},\qquad x\in (\K_0,\mathcal{H}_0)_{\vartheta,\infty}.
\end{align*}
\end{thm}

By \eqref{Goku}, \eqref{Gokulip}, Theorem \ref{Clip} and Theorem \ref{classic} (with $\vartheta=\alpha$, $\K_0,\K_1,\mathcal{H}_1={\rm BUC}(E)$, $\mathcal{H}_0={\rm Lip}_b(E)$, $\vartheta=\alpha$ and $T=\J_R^j P(t)$ for $j=1,2,3$) we obtain the following proposition which is the main result of this section.

\begin{prop}
Assume that Hypotheses \ref{RDS1} hold true and let $\alpha\in (0,1)$. There exists $c_\alpha>0$ such that for any $j=1,2,3$ and $\varphi\in C_b^\alpha(E)$ we have
\begin{align}\label{Goku_alpha}
\|\J^jP(t)\varphi(x)\|_{\mathcal{L}^{(j)}(E;\R)}\leq c_\alpha (\min\{1,t\})^{-(j-\alpha)(1+\gamma)/2}\norm{\varphi}_{C_b^\alpha(E)}.
\end{align}
\end{prop}
%
%
%

\begin{rmk}
If Hypotheses \ref{RDS1} holds true, then it is possible to prove that for any $x\in L^2(\mathcal{O})$ the stochastic partial differential equation \eqref{eqF02} has a unique generalized mild solution $\mi{X}$ (see \cite[Proposition 7.1.2]{CER1}). Moreover using the generalized mild solution $\mi{X}$ it is possible to define a transition semigroup $\{P^{L^2(\mathcal{O})}(t)\}_{t\geq 0}$ associated to \eqref{eqF02} on ${\rm BUC}(L^2(\mathcal{O}))$, and also $\{P^{L^2(\mathcal{O})}(t)\}_{t\geq 0}$ is weakly continuous (see \cite[Section 7.1]{CER1}). However we are not aware of the existence of estimates similar to \eqref{Goku} and \eqref{GokuC} for $\{P^{L^2(\mathcal{O})}(t)\}_{t\geq 0}$.
\end{rmk}

\section{Proofs of the main results}

This section is devoted to the proof the main results of this paper.

\subsection{Stationary equation}\label{proofstaz}
We start by proving Schauder regularity results for the function $u$ introduced in \eqref{uintro}.

\subsubsection{The case $f\in{\rm BUC}(E)$} First of all we analize the regularity of the function $u$ when $f$ is just a real-valued, bounded and uniformly continuous function.

\begin{proof}[Proof of Theorem \ref{THM_reg_BUC}]
We start by proving that if $\gamma\in[0,1)$, then $u\in{\rm BUC}^1(E)$. By \eqref{Goku} (with $j=1$) we get that for any $x,h\in E$ and $t>0$
\begin{align}
|P(t)f(x+h)-P(t)f(x)-\D P(t)f(x)h|&=\abs{\int_0^1(\D P(t)f(x+\sigma h)-\D P(t)f(x))hd\sigma}\notag\\
&\leq 2C_1 (\min\{1,t\})^{-(1+\gamma)/2}\|h\|_E\norm{f}_{\infty}.\label{save1}
\end{align}
Hence, by \eqref{save1} and the dominated convergence theorem we get that $u$ is Fr\'echet differentiable and for every $\lambda,t>0$ and $x\in E$ it holds
\begin{align*}
\D u(x)&=\int_0^{+\infty}e^{-\lambda t}\D P(t)f(x)dt.
\end{align*}
Using the same arguments used in the proof of \eqref{save1} we obtain 
\begin{align}\label{exam1}
\|\D u(x)\|_{\mathcal{L}(E;\R)}&\leq C_1\pa{\frac{2}{1-\gamma}+\frac{1}{\lambda}}\|f\|_{\infty}.
\end{align}
The fact that $u\in {\rm BUC}^1(E)$ follows by Proposition \ref{Zone400} (with $\mathcal{M}=E$, $d_{\mathcal{M}}=\norm{\cdot}_E$, $Y=(0,+\infty)$, $\mu$ is the measure $e^{-\lambda t}dt$ and $Z$ is $\mathcal{L}(E;\R)$).

We can now show \eqref{THM_reg_BUC1}. If $x,h\in E$ are such that $\|h\|_E\geq 1$, then by \eqref{exam1}, it holds
\begin{align}
\|\D u(x+2h)-2\D u(x+h)-\D u(x)\|_{\mathcal{L}(E;\R)}&\leq 4C_1\pa{\frac{2}{1-\gamma}+\frac{1}{\lambda}}\|f\|_{\infty}\notag\\
&\leq 4C_1\pa{\frac{2}{1-\gamma}+\frac{1}{\lambda}}\|h\|_E\|f\|_{\infty}.\label{caso>1}
\end{align}
Now we are going to study the case $\|h\|_E<1$. Consider the functions $a,b:E\ra\R$ defined, for any $x,h\in E$, as
\begin{align*}
a_1(x):=\int_0^{\norm{h}_E^2}e^{-\lambda t} \D P(t)f(x)dt;\qquad b_1(x):=\int_{\norm{h}_E^2}^{+\infty}e^{-\lambda t}\D P(t)f(x)dt.
\end{align*}
For any $x,h\in E$ with $\|h\|_E<1$ and $t>0$, by \eqref{Goku} (with $j=1$), we have 
\begin{align}
\|a_1(x+2h)&-2a_1(x+h)+a_1(x)\|_{\mathcal{L}(E;\R)}\notag\\
&\leq\int_0^{\|h\|_E}e^{-\lambda t}\|\D P(t)f(x+2h)-2\D P(t)f(x+h)+\D P(t)f(x)\|_{\mathcal{L}(E;\R)}dt\notag\\
&\leq 8\|h\|_E\|f\|_\infty.\label{BA_1}
\end{align}
Before proceeding we need some intermediate estimates. First of all observe that for every $x,h\in E$ and $t>0$
\begin{align}
\D P(t)f(x+2h)-\D P(t)f(x+h) &=\int_0^1\D^2_RP(t)f(x+(1+\sigma)h)(h,\cdot)d\sigma;\label{BA_2}\\
\D P(t)f(x+h)-\D P(t)f(x) &=\int_0^1\D^2_R P(t)f(x+\sigma h)(h,\cdot)d\sigma.\label{BA_3}
\end{align}
So combining \eqref{BA_2} and \eqref{BA_3} and we get
\begin{align}\label{BA_5}
\D P(t)f(x+2h)-2\D P(t)f(x+h)+\D P(t)f(x)=\int_0^1\int_0^1\D^3P(t)f(x+(\tau+\sigma)h)(h,h,\cdot)d\tau d\sigma.
\end{align}
By \eqref{Goku} (with $j=3$), \eqref{BA_5} and recalling that there exists $K(\lambda)>0$ such that for any $t>0$, it holds $t^{3/2} e^{-\lambda t}\leq K(\lambda)$, we get 
\begin{align}
\|b_1(x+2h)&-2b_1(x+h)+b_1(x)\|_{\mathcal{L}(E;\R)}\notag\\
&\leq\int_{\|h\|_E^2}^{+\infty}e^{-\lambda t}\|\D P(t)f(x+2h)-2\D P(t)f(x+h)+\D P(t)f(x)\|_{\mathcal{L}(E;\R)} dt\notag\\
&\leq 2C_3(1+K(\lambda))\|h\|_E\Vert f\Vert_{\infty}.\label{BA_6}
\end{align}
Combining \eqref{caso>1}, \eqref{BA_1} and \eqref{BA_6} we get \eqref{THM_reg_BUC1}. 

To prove \eqref{THM_reg_BUC2}, let $x,h\in E$ be such that $\|h\|_E\geq 1$. By \eqref{exam1}, it holds
\begin{align}
\|\D u(x+h)-\D u(x)\|_{\mathcal{L}(E;\R)}&\leq 2C_1\pa{\frac{2}{1-\gamma}+\frac{1}{\lambda}}\|f\|_{\infty}\leq 2C_1\pa{\frac{2}{1-\gamma}+\frac{1}{\lambda}}\|h\|^{(1-\gamma)/(1+\gamma)}_E\|f\|_{\infty}.\label{caso>1,2}
\end{align}
Now we are going to study the case $\|h\|_E<1$. Let $x,h\in E$ be such that $\|h\|_E<1$, and consider the functions $a_{1,\gamma},b_{1,\gamma}:E\ra\R$ defined as
\begin{align*}
a_{1,\gamma}(x):=\int_0^{\|h\|_E^{2/(1+\gamma)}}e^{-\lambda t}\D P(t)f(x)dt;\qquad b_{1,\gamma}(x):=\int_{\|h\|_E^{2/(1+\gamma)}}^{+\infty}e^{-\lambda t}\D P(t)f(x)dt.
\end{align*}
Observe that by \eqref{Goku} (with $j=1$) it holds
\begin{align}
\|a_{1,\gamma}(x+h)-a_{1,\gamma}(x)\|_{\mathcal{L}(E;\R)}&\leq \int_0^{\|h\|_E^{2/(1+\gamma)}}e^{-\lambda t}\|\D P(t)f(x+h)-\D P(t)f(x)\|_{\mathcal{L}(E;\R)}dt\notag\\
&\leq 2C_1\int_0^{\|h\|_E^{2/(1+\gamma)}}t^{-(1+\gamma)/2} dt=\frac{2C_1}{1-\gamma}\|h\|^{(1-\gamma)/(1+\gamma)}\|f\|_\infty. \label{BA_3/2}
\end{align}
Furthermore by \eqref{Goku} (with $j=2$) and the fact that there exists $K(\gamma,\lambda)>0$ such that, for any $t>0$, it holds $t^{(1+\gamma)}e^{-\lambda t}\leq K(\gamma,\lambda)$ we get
\begin{align}
\|b_{1,\gamma}(x+h)-b_{1,\gamma}(x)\|_{\mathcal{L}(E;\R)}&\leq \int_{\|h\|_E^{2/(1+\gamma)}}^{+\infty}e^{-\lambda t}\|\D P(t)f(x+h)-\D P(t)f(x)\|_{\mathcal{L}(E;\R)}dt\notag\\
&\leq \int_{\|h\|_E^{2/(1+\gamma)}}^{+\infty}e^{-\lambda t}\norm{\int_0^1\D^2 P(t)f(x+\sigma h)(h,\cdot) d\sigma}_{\mathcal{L}(E;\R)}dt\notag\\
&\leq C_2\|h\|_E\|f\|_\infty\pa{\int_{\|h\|_E^{2/(1+\gamma)}}^{+\infty}e^{-\lambda t}(\min\{1,t\})^{-(1+\gamma)}dt}\notag\\
&\leq C_2\|h\|_E\|f\|_\infty\pa{\int_{\|h\|_E^{2/(1+\gamma)}}^{+\infty}t^{-(1+\gamma)/2}dt+\int_{1}^{+\infty}e^{-\lambda t}dt}\notag\\
&=\frac{2C_2}{\gamma}(1+K(\gamma,\lambda))\|h\|_E^{(1-\gamma)/(1+\gamma)}\|f\|_\infty.\label{BA_4}
\end{align}
Now combining \eqref{caso>1,2}, \eqref{BA_3/2} and \eqref{BA_4} we get \eqref{THM_reg_BUC2}. \eqref{THM_reg_BUC3} follows using arguments similar to those used to prove \eqref{THM_reg_BUC1}, so its proof is left to the reader.
\end{proof}

\subsubsection{The case $f\in C_b^\alpha(E)$}
Now we can prove Theorem \ref{Thm_Schauder_Rd}.

\begin{proof}[Proof of Theorem \ref{Thm_Schauder_Rd}]
We start by proving that if $\gamma\in[0,1]$, then $u\in{\rm BUC}^1(E)$. By \eqref{Goku_alpha} (with $j=1$) we get that for any $x,h\in E$ and $t>0$
\begin{align}
|P(t)f(x+h)-P(t)f(x)-\D P(t)f(x)h|&=\abs{\int_0^1(\D P(t)f(x+\sigma h)-\D P(t)f(x))hd\sigma}\notag\\
&\leq 2c_\alpha (\min\{1,t\})^{-(1-\alpha)(1+\gamma)/2}\|h\|_E\norm{f}_{C_b^\alpha(E)}.\label{AMA1}
\end{align}
Hence, by \eqref{AMA1} and the dominated convergence theorem we get that $u$ is Fr\'echet differentiable and for every $\lambda,t>0$ and $x\in E$ it holds
\begin{align*}
\D u(x)&=\int_0^{+\infty}e^{-\lambda t}\D P(t)f(x)dt.
\end{align*}
Using the same arguments used in the proof of \eqref{AMA1} we obtain 
\begin{align*}
\|\D_R u(x)\|_{\mathcal{L}(E;\R)}&\leq c_\alpha\pa{\frac{2}{2-(1-\alpha)(1+\gamma)}+\frac{1}{\lambda}}\|f\|_{C_b^\alpha(E)}.
\end{align*}
The fact that $u\in {\rm BUC}^1(E)$ follows by Proposition \ref{Zone400} (with $\mathcal{M}=E$, $d_{\mathcal{M}}=\norm{\cdot}_E$, $Y=(0,+\infty)$, $\mu$ is the measure $e^{-\lambda t}dt$ and $Z$ is $\mathcal{L}(E;\R)$).

Now we show that if $\gamma\in[0,1)$ and $\alpha>\frac{2\gamma}{1+\gamma}$, then $u\in{\rm BUC}^2(E)$. By \eqref{Goku_alpha} (with $j=2$) we get that for any $x,h\in E$ and $t>0$
\begin{align}
\|\D P(t)f(x+h)-&\D P(t)f(x)-\D^2 P(t)f(x)(h,\cdot)\|_{\mathcal{L}(E;\R)}\notag\\
&=\norm{\int_0^1(\D^2_RP(t)f(x+\sigma h)-\D^2_RP(t)f(x))(h,\cdot)d\sigma}_{\mathcal{L}(E;\R)}\notag\\
&\leq 2c_\alpha (\min\{1,t\})^{-(2-\alpha)(1+\gamma)/2}\|h\|_E\norm{f}_{C_b^\alpha(E)}.\label{AMA2}
\end{align}
Hence, by \eqref{AMA2} and the dominated convergence theorem we get that $u$ is Fr\'echet differentiable and for every $\lambda,t>0$ and $x\in E$ it holds
\begin{align*}
\D^2 u(x)&=\int_0^{+\infty}e^{-\lambda t}\D^2 P(t)f(x)dt.
\end{align*}
Using the same arguments used in the proof of \eqref{AMA2} we obtain 
\begin{align*}
\|\D^2_R u(x)\|_{\mathcal{L}^{(2)}(E;\R)}&\leq c_\alpha\pa{\frac{2}{2-(2-\alpha)(1+\gamma)}+\frac{1}{\lambda}}\|f\|_{C_b^\alpha(E)}.
\end{align*}
The fact that $u\in {\rm BUC}^2(E)$ follows by Proposition \ref{Zone400} (with $\mathcal{M}=E$, $d_{\mathcal{M}}=\norm{\cdot}_E$, $Y=(0,+\infty)$, $\mu$ is the measure $e^{-\lambda t}dt$ and $Z$ is $\mathcal{L}^{(2)}(E;\R)$).

Now we can prove \eqref{Thm_Schauder_Rd1}. Let $\lambda>0$ and $x,h\in E$ such that $\|h\|_E<1$ and consider the functions
\begin{align*}
a_{1,\gamma}(x):=\int_0^{\|h\|_E^{2/(1+\gamma)}}e^{-\lambda t}\D P(t)f(x)dt;\qquad b_{1,\gamma}(x):=\int^{+\infty}_{\|h\|_E^{2/(1+\gamma)}}e^{-\lambda t}\D P(t)f(x)dt.
\end{align*}
Now by \eqref{Goku_alpha} (with $j=1$) it holds
\begin{align}
\|a_{1,\gamma}(x+h)-a_{1,\gamma}(x)\|_{\mathcal{L}(E;\R)}&\leq 2c_\alpha\|f\|_{C_b^\alpha(E)}\int_0^{\|h\|_E^{2/(1+\gamma)}}t^{-(1-\alpha)(1+\gamma)/2}dt\notag\\
& =\frac{4c_\alpha}{2-(1-\alpha)(1+\gamma)}\|h\|_E^{\alpha+(1-\gamma)/(1+\gamma)}\|f\|_{C_b^\alpha(E)}.\label{community1}
\end{align}
Now using \eqref{Goku_alpha} (with $j=2$) and arguing in the same way as in the proof of Theorem \ref{THM_reg_BUC}, we get that there exists $K=K(\lambda,\gamma,\alpha)>0$ such that
\begin{align}
\|b_{1,\gamma}(x+h)-b_{1,\gamma}(x)\|_{\mathcal{L}(E;\R)}&\leq K\|h\|_E\|f\|_{C_b^\alpha(E)}\int_{\|h\|_E^{2/(1+\gamma)}}^{+\infty}t^{-(2-\alpha)(1+\gamma)/2}dt\notag\\
&=\frac{2K}{(2-\alpha)(1+\gamma)-2}\|h\|_E^{\alpha+(1-\gamma)/(1+\gamma)}\|f\|_{C_b^\alpha(E)}.\label{community2}
\end{align}
Combining \eqref{community1} and \eqref{community2} we get \eqref{Thm_Schauder_Rd1}, indeed the case $\|h\|_E\geq 1$ can be obtained using arguments similar to those used in the proof of \eqref{caso>1,2}. The proofs of \eqref{Thm_Schauder_Rd2} and \eqref{Thm_Schauder_Rd3} are similar to the one in Theorem \ref{THM_reg_BUC} and are left to the reader.
\end{proof}

\subsection{The evolution equation}\label{proofevol}
In this section we study the regularity of the mild solution of \eqref{evol_prob}, namely the function
\begin{align}\label{mild_evol}
v(t,x):= P(t)f(x)+\int_0^t P(t-s)g(s,\cdot)(x)ds,\qquad T>0,\ t\in[0,T],\ x\in E.
\end{align}
To do so we will analyze the regularity of the two summands on the right hand side of \eqref{mild_evol}. Before all of that we need to introduce the spaces we will use in this section.

\begin{defi}
Let $Y$ be a Banach space with norm $\norm{\cdot}_Y$. For any $\alpha\in(0,1)$, $k=0,1,2,\ldots$ and $T>0$ we define ${\rm BUC}^{0,k+\alpha}([0,T]\times E; Y)$ as the set of continuous functions $g:[0,T]\times E\ra Y$, that are separately uniformly continuous and such that 
\begin{align*}
\|g\|_{{\rm BUC}^{0,k+\alpha}([0,T]\times E;Y)}:=\sup_{t\in[0,T]}\|g(t,\cdot)\|_{{\rm BUC}^{k+\alpha}(E;Y)}<+\infty
\end{align*}
If $Y=\R$ we write ${\rm BUC}^{0,k+\alpha}([0,T]\times E)$.
\end{defi}

\noindent For any $T>0$, $\alpha\in(0,1)$ and $k=0,1,2,\ldots$ the space ${\rm BUC}^{0,k+\alpha}([0,T]\times E; Y)$ is a Banach space if endowed with the norm $\norm{\cdot}_{{\rm BUC}^{0,k+\alpha}([0,T]\times E;Y)}$.

The following result says that the function $P(t)f(x)$ has the ``same'' regularty of $f$.

\begin{prop}
Assume Hypotheses \ref{RDS1} hold true. For every $T>0$, $\alpha\in(0,1)$, $k=0,1,2,3$ and $f\in C_b^{k+\alpha}(E)$, the map $(t,x)\mapsto P(t)f(x)$ belongs to ${\rm BUC}^{0,k+\alpha}([0,T]\times E)$.
\end{prop}

\begin{proof}
We just show the case $k=0$, since the other cases follow by similar arguments. Observe that the fact that, for every $t\in[0,T]$, the map $x\mapsto P(t)f(x)$ is uniformly continuous follows by Proposition \ref{Zone400} (with $\mathcal{M}=E$, $d_{\mathcal{M}}=\norm{\cdot}_E$, $Y=\Omega$, $\mu=\mathbb{P}$ and $Z=\R$) and the boundedness and uniform continuity of $f$, while the uniform continuity of $t\mapsto P(t)f(x)$ follows by its continuity and the compactness of $[0,T]$. Let $h(t,x):=P(t)f(x)$ for $t\in[0,T]$ and $x\in E$. By Proposition \ref{daviderompeilcazzo} we obtain
\begin{align*}
\|h\|_{{\rm BUC}^{0,\alpha}([0,T]\times E)}&=\sup_{t\in [0,T]}(\|P(t)f\|_\infty+[P(t)f]_{\alpha})\leq \max\{1,e^{\alpha\eta T}\}\|f\|_{{\rm BUC}^\alpha(E)},
\end{align*}
hence the map $(t,x)\mapsto P(t)f(x)$ belongs to ${\rm BUC}^{0,\alpha}([0,T]\times E)$. This concludes the proof.
\end{proof}

The following theorem study the regularity of the second summand in \eqref{mild_evol} in the case $g\in {\rm BUC}^{0,0}([0,T]\times E)$.

\begin{thm}\label{THM_evol_reg_BUC}
Assume Hypotheses \ref{RDS1} hold true. Let $T>0$, $g\in {\rm BUC}^{0,0}([0,T]\times E)$ and consider the function
\begin{align*}
v_2(t,x):=\int_0^tP(s)g(t-s,\cdot)(x)ds,\qquad t\in[0,T],\ x\in E.
\end{align*}
It holds
\begin{enumerate}[\rm (i)]
\item\label{THM_evol_reg_BUC1} If $\gamma=0$, then $v_2\in {\rm BUC}^{0,1}([0,T]\times E)$ and for every $t\in[0,T]$ the function $x\mapsto\D v_2(t,x)$ belongs to $\mathcal{Z}^1(E;\mathcal{L}(E;\R))$. Moreover there exists $C=C(T)>0$ such that
\begin{align*}
\sup_{t\in[0,T]}\|\D v_2(t,\cdot)\|_{\mathcal{Z}^1(E;\mathcal{L}(E;\R))}\leq C\|g\|_{{\rm BUC}^{0,0}([0,T]\times E)}.
\end{align*}

\item\label{THM_evol_reg_BUC2} If $\gamma\in(0,1)$, then $v_2$ belongs to ${\rm BUC}^{0,1+(1-\gamma)/(1+\gamma)}([0,T]\times E)$ and there exists $C=C(\gamma,T)>0$ such that
\begin{align*}
\|v_2\|_{{\rm BUC}^{0,1+(1-\gamma)/(1+\gamma)}([0,T]\times E)}\leq C\|g\|_{{\rm BUC}^{0,0}([0,T]\times E)}.
\end{align*}

\item\label{THM_evol_reg_BUC3} If $\gamma=1$, then for every $t\in[0,T]$ the function $x\mapsto v_2(t,x)$ belongs to $\mathcal{Z}^1(E;\R)$ and there exists $C=C(T)>0$ such that
\begin{align*}
\sup_{t\in[0,T]}\|v_2(t,\cdot)\|_{\mathcal{Z}^1(E;\R)}\leq C\|g\|_{{\rm BUC}^{0,0}([0,T]\times E)}.
\end{align*}
\end{enumerate}
\end{thm}

\begin{proof}
Observe that for every $\gamma\in[0,1]$, by Proposition \ref{daviderompeilcazzo} it holds $\sup_{t\in[0,T]}\|v_2(t,\cdot)\|_\infty\leq T\|g\|_{{\rm BUC}^{0,0}([0,T]\times E)}$. If $\gamma\in[0,1)$ then by \eqref{Goku} (with $j=1$) we get that for any $x,h\in E$ and $t\in[0,T]$
\begin{align}
&\left|v_2(t,x+h)-v_2(t,x)-\int_0^t\D P(s)g(t-s,\cdot)(x)hds\right|\notag\\
&=\abs{\int_0^t\int_0^1(\D P(s)g(t-s,\cdot)(x+\sigma h)-\D P(s)g(t-s,\cdot)(x))hd\sigma ds}\notag\\
&\leq 2C_1 \pa{\int_0^t(\min\{1,s\})^{-(1+\gamma)/2}ds}\|h\|_E\norm{g}_{{\rm BUC}^{0,0}([0,T]\times E)}\notag\\
&=2C_1 \pa{\frac{2}{1-\gamma}+(T-1)\chi_{(1,+\infty)}(T)}\|h\|_E\norm{g}_{{\rm BUC}^{0,0}([0,T]\times E)}.\label{saveing1}
\end{align}
So by \eqref{saveing1}, for any $t\in[0,T]$, the function $x\mapsto v_2(t,x)$ is Fr\'echet differentiable and by Proposition \ref{Zone400} (with $\mathcal{M}=E$, $d_{\mathcal{M}}=\norm{\cdot}_E$, $Y=(0,t)$, $\mu$ is the measure $dt$ and $Z$ is $\mathcal{L}(E;\R)$) belongs to ${\rm BUC}^1(E)$. Using the same arguments as in \eqref{saveing1} we get
\begin{align*}
\sup_{t\in [0,T]}\|\D v_2(t,\cdot)\|_{\mathcal{L}(E;\R)}\leq C_1 \pa{\frac{2}{1-\gamma}+(T-1)\chi_{(1,+\infty)}(T)}\|h\|_E\norm{g}_{{\rm BUC}^{0,0}([0,T]\times E)}.
\end{align*}
So $v_2$ belongs to ${\rm BUC}^{0,1}([0,T]\times E)$.

Now we prove \eqref{THM_evol_reg_BUC1}. For any $t\in[0,T]$ and $x,h\in E$ with $\|h\|_E<1$ we consider
\begin{align*}
a_!(t,x):=\int_0^{\min\{t,\|h\|_E^2\}}\D P(s)g(t-s,\cdot)(x)ds,\qquad b_1(t,x):=\int_{\min\{t,\|h\|_E^2\}}^t \D P(s)g(t-s,\cdot)(x)ds.
\end{align*}
Now observe that by \eqref{Goku} (with $j=1$) we get
\begin{align}
\|a_1(t,x+2h)-2a_1(t,x+h)+a_1(t,x)\|_{\mathcal{L}(E;\R)} &\leq 4C_1 \pa{\int_0^{\min\{t,\|h\|_E^2\}}s^{-1/2}ds}\|g\|_{{\rm BUC}^{0,0}([0,T]\times E)}\notag\\
& \leq 8C_1\|h\|_E\|g\|_{{\rm BUC}^{0,0}([0,T]\times E)}.\label{HULK1}
\end{align}
Observe that $b_1(t,x+2h)-2b_1(t,x+h)+b_1(t,x)=0$ if $\|h\|_E^2\geq t$. So if $\|h\|_E^2<t$, by \eqref{Goku} (with $j=3$), we get 
\begin{align}
\|b_1(t,x+2h)&-2b_1(t,x+h)+b_1(t,x)\|_{\mathcal{L}(E;\R)}\notag\\ 
&=\norm{\int_{\min\{t,\|h\|_E^2\}}^t\int_0^1\int_0^1\D^3P(s)g(t-s,\cdot)(x+(\tau+\sigma)h)(h,h,\cdot)d\tau d\sigma ds}_{\mathcal{L}(E;\R)}\notag\\
&\leq C_3\|h\|_E^2\pa{\int_{\|h\|_E^2}^1s^{-3/2}ds+\chi_{(1,+\infty)}(T)\int_1^Tds}\|g\|_{{\rm BUC}^{0,0}([0,T]\times E)}\notag\\
&\leq C_3\|h\|_E^2\pa{\int_{\|h\|_E^2}^{+\infty}s^{-3/2}ds+\chi_{(1,+\infty)}(T)\int_1^T\frac{1}{\|h\|_E}ds}\|g\|_{{\rm BUC}^{0,0}([0,T]\times E)}\notag\\
&= C_3\pa{2+(T-1)\chi_{(1,+\infty)}(T)}\|h\|_E\|g\|_{{\rm BUC}^{0,0}([0,T]\times E)}.\label{HULK2}
\end{align}
Combining \eqref{HULK1} and \eqref{HULK2} we get \eqref{THM_evol_reg_BUC1}, indeed the case $\|h\|_E\geq 1$ can be obtained using arguments similar to those used in the proof of \eqref{caso>1}.

It is now time to prove \eqref{THM_evol_reg_BUC2}. For any $t\in[0,T]$ and $x,h\in E$ with $\|h\|_E<1$ we consider
\begin{align*}
a_{1,\gamma}(t,x)&:=\int_0^{\min\{t,\|h\|_E^{2/(1+\gamma)}\}}\D P(s)g(t-s,\cdot)(x)ds,\\
b_{1,\gamma}(t,x)&:=\int_{\min\{t,\|h\|_E^{2/(1+\gamma)}\}}^t \D P(s)g(t-s,\cdot)(x)ds.
\end{align*}
Now observe that by \eqref{Goku} (with $j=1$) we get
\begin{align}
\|a_{1,\gamma}(t,x+h)-a_{1,\gamma}(t,x)\|_{\mathcal{L}(E;\R)} &\leq 2C_1 \pa{\int_0^{\min\{t,\|h\|_E^{2/(1+\gamma)}\}}s^{-(1+\gamma)/2}ds}\|g\|_{{\rm BUC}^{0,0}([0,T]\times E)}\notag\\
& \leq \frac{4 C_1}{1-\gamma}\|h\|_E^{(1-\gamma)/(1+\gamma)}\|g\|_{{\rm BUC}^{0,0}([0,T]\times E)}.\label{HULK3}
\end{align}
Observe that $b_{1,\gamma}(t,x+h)-b_{1,\gamma}(t,x)=0$ if $\|h\|_E^{2/(1+\gamma)}\geq t$. So if $\|h\|_E^{2/(1+\gamma)}<t$, by \eqref{Goku} (with $j=2$), we get 
\begin{align}
&\|b_{1,\gamma}(t,x+h)-b_{1,\gamma}(t,x)\|_{\mathcal{L}(E;\R)}\notag\\
&=\norm{\int_{\min\{t,\|h\|_E^{2/(1+\gamma)}\}}^t\int_0^1\D^2P(s)g(t-s,\cdot)(x+\sigma h)(h,\cdot) d\sigma ds}_{\mathcal{L}(E;\R)}\notag\\
&\leq C_2\|h\|_E\pa{\int_{\|h\|_E^{2/(1+\gamma)}}^1s^{-(1+\gamma)}ds+\chi_{(1,+\infty)}(T)\int_1^Tds}\|g\|_{{\rm BUC}^{0,0}([0,T]\times E)}\notag\\
&\leq C_2\|h\|_E\pa{\int_{\|h\|_E^{2/(1+\gamma)}}^{+\infty}s^{-(1+\gamma)}ds+\chi_{(1,+\infty)}(T)\int_1^T\frac{1}{\|h\|_E^{2\gamma/(1+\gamma)}}ds}\|g\|_{{\rm BUC}^{0,0}([0,T]\times E)}\notag\\
&= C_2\pa{\frac{1}{\gamma}+(T-1)\chi_{(1,+\infty)}(T)}\|h\|_E^{(1-\gamma)/(1+\gamma)}\|g\|_{{\rm BUC}^{0,0}([0,T]\times E)}.\label{HULK4}
\end{align}
Combining \eqref{HULK3} and \eqref{HULK4} we get \eqref{THM_evol_reg_BUC2}, indeed the case $\|h\|_E\geq 1$ can be obtained using arguments similar to those used in the proof of \eqref{caso>1,2}. The proof of \eqref{THM_evol_reg_BUC3} follows from the same arguments as the proof of \eqref{THM_evol_reg_BUC1}.
%
\end{proof}

We can now state the following result which is in the same spirit of the results of \cite{KCL75,KCL80}, with the twist presented in this paper.

\begin{thm}\label{Thm_evol_Schauder_Rd}
Assume Hypotheses \ref{RDS1} hold true. Let $T>0$, $\alpha\in(0,1)$, $g\in {\rm BUC}^{0,\alpha}([0,T]\times E)$ and consider the function
\begin{align*}
v_2(t,x):=\int_0^tP(s)g(t-s,\cdot)(x)ds,\qquad t\in[0,T],\ x\in E.
\end{align*}
It holds
\begin{enumerate}[\rm (i)]
\item\label{Thm_evol_Schauder_Rd1} If $\gamma\in (0,1]$ and $0<\alpha<\frac{2\gamma}{1+\gamma}$ then $v_2$ belongs to ${\rm BUC}^{0,1+\alpha+(1-\gamma)/(1+\gamma)}([0,T]\times E)$ and there exists $C=C(T,\gamma,\alpha)>0$ such that
\begin{align*}
\|v_2\|_{{\rm BUC}^{0,1+\alpha+(1-\gamma)/(1+\gamma)}([0,T]\times E)}\leq C\|g\|_{{\rm BUC}^{0,\alpha}([0,T]\times E)}.
\end{align*}

\item\label{Thm_evol_Schauder_Rd2} If $\gamma\in(0,1)$ and $\alpha=\frac{2\gamma}{1+\gamma}$, then $v_2$ belongs to ${\rm BUC}^{0,1}([0,T]\times E)$ and, for every $t\in[0,T]$, the map $x\mapsto\D v_2(t,x)$ belongs to $\mathcal{Z}^1(E;\mathcal{L}(E;\R))$ and there exists $C=C(T,\gamma)>0$ such that
\begin{align*}
\sup_{t\in[0,T]}\|\D v_2(t,\cdot)\|_{\mathcal{Z}^1(E;\mathcal{L}(E;\R))}\leq C\|g\|_{{\rm BUC}^{0,2\gamma/(1+\gamma)}([0,T]\times E)}.
\end{align*}

\item\label{Thm_evol_Schauder_Rd3} If $\gamma\in[0,1)$ and $\alpha>\frac{2\gamma}{1+\gamma}$ then $v_2$ belongs to ${\rm BUC}^{0,2+\alpha-2\gamma/(1+\gamma)}([0,T]\times E)$ and there exists $C=C(T,\gamma,\alpha)>0$ such that
\begin{align*}
\|v_2\|_{{\rm BUC}^{0,2+\alpha-2\gamma/(1+\gamma)}([0,T]\times E)}\leq C\|g\|_{{\rm BUC}^{0,\alpha}([0,T]\times E)}.
\end{align*}
\end{enumerate}
\end{thm}

\begin{proof}
Observe that for every $\gamma\in[0,1]$, by Proposition \ref{daviderompeilcazzo} it holds $\sup_{t\in[0,T]}\|v_2(t,\cdot)\|_\infty\leq T\|g\|_{{\rm BUC}^{0,0}([0,T]\times E)}$. If $\gamma\in[0,1]$ then by \eqref{Goku_alpha} (with $j=1$) we get that for any $x,h\in E$ and $t\in[0,T]$
\begin{align}
&\left|v_2(t,x+h)-v_2(t,x)-\int_0^t\D P(s)g(t-s,\cdot)(x)hds\right|\notag\\
&=\abs{\int_0^t\int_0^1(\D P(s)g(t-s,\cdot)(x+\sigma h)-\D P(s)g(t-s,\cdot)(x))hd\sigma ds}\notag\\
&\leq 2c_\alpha \pa{\int_0^t(\min\{1,s\})^{-(1-\alpha)(1+\gamma)/2}ds}\|h\|_E\norm{g}_{{\rm BUC}^{0,\alpha}([0,T]\times E)}\notag\\
&=2c_\alpha \pa{\frac{2}{2-(1-\alpha)(1+\gamma)}+(T-1)\chi_{(1,+\infty)}(T)}\|h\|_E\norm{g}_{{\rm BUC}^{0,\alpha}([0,T]\times E)}.\label{saveing2}
\end{align}
So by \eqref{saveing2}, for any $t\in[0,T]$, the function $x\mapsto v_2(t,x)$ is Fr\'echet differentiable and by Proposition \ref{Zone400} (with $\mathcal{M}=E$, $d_{\mathcal{M}}=\norm{\cdot}_E$, $Y=(0,t)$, $\mu$ is the measure $dt$ and $Z$ is $\mathcal{L}(E;\R)$) it belongs to ${\rm BUC}^1(E)$. Using the same arguments as in \eqref{saveing2} we get
\begin{align*}
\sup_{t\in [0,T]}\|\D v_2(t,\cdot)\|_{\mathcal{L}(E;\R)}\leq c_\alpha \pa{\frac{2}{2-(1-\alpha)(1+\gamma)}+(T-1)\chi_{(1,+\infty)}(T)}\norm{g}_{{\rm BUC}^{0,\alpha}([0,T]\times E)}.
\end{align*}
So $v_2$ belongs to ${\rm BUC}^{0,1}([0,T]\times E)$.

Now if $\gamma\in[0,1)$ and $\alpha>\frac{2\gamma}{1+\gamma}$ then by \eqref{Goku_alpha} (with $j=2$) we get that for any $x,h\in E$ and $t\in[0,T]$
\begin{align}
&\left\|\D v_2(t,x+h)-\D v_2(t,x)-\int_0^t(\D^2 P(s)g(t-s,\cdot)(x))(h,\cdot)ds\right\|_{\mathcal{L}(E;\R)}\notag\\
&=\norm{\int_0^t\int_0^1(\D^2 P(s)g(t-s,\cdot)(x+\sigma h)-\D^2 P(s)g(t-s,\cdot)(x))(h,\cdot)d\sigma ds}_{\mathcal{L}(E;\R)}\notag\\
&\leq 2c_\alpha \pa{\int_0^t(\min\{1,s\})^{-(2-\alpha)(1+\gamma)/2}ds}\|h\|_E\norm{g}_{{\rm BUC}^{0,\alpha}([0,T]\times E)}\notag\\
&=2c_\alpha \pa{\frac{2}{2-(2-\alpha)(1+\gamma)}+(T-1)\chi_{(1,+\infty)}(T)}\|h\|_E\norm{g}_{{\rm BUC}^{0,\alpha}([0,T]\times E)}.\label{saveing3}
\end{align}
So by \eqref{saveing3}, for any $t\in[0,T]$, the function $x\mapsto v_2(t,x)$ is two times Fr\'echet differentiable and by Proposition \ref{Zone400} (with $\mathcal{M}=E$, $d_{\mathcal{M}}=\norm{\cdot}_E$, $Y=(0,t)$, $\mu$ is the measure $dt$ and $Z$ is $\mathcal{L}^{(2)}(E;\R)$) it belongs to ${\rm BUC}^2(E)$. Using the same arguments as in \eqref{saveing3} we get
\begin{align*}
\sup_{t\in [0,T]}\|\D^2 v_2(t,\cdot)\|_{\mathcal{L}^{(2)}(E;\R)}\leq c_\alpha \pa{\frac{2}{2-(2-\alpha)(1+\gamma)}+(T-1)\chi_{(1,+\infty)}(T)}\norm{g}_{{\rm BUC}^{0,\alpha}([0,T]\times E)}.
\end{align*}
So $v_2$ belongs to ${\rm BUC}^{0,2}([0,T]\times E)$.

We only prove \eqref{Thm_evol_Schauder_Rd1}, since the proofs of \eqref{Thm_evol_Schauder_Rd2} and \eqref{Thm_evol_Schauder_Rd3} are similar to the one contained in Theorem \ref{THM_evol_reg_BUC}. For any $t\in[0,T]$ and $x,h\in E$ with $\|h\|_E<1$ we consider
\begin{align*}
a_{1,\gamma}(t,x)&:=\int_0^{\min\{t,\|h\|_E^{2/(1+\gamma)}\}}\D P(s)g(t-s,\cdot)(x)ds,\\
b_{1,\gamma}(t,x)&:=\int_{\min\{t,\|h\|_E^{2/(1+\gamma)}\}}^t \D P(s)g(t-s,\cdot)(x)ds.
\end{align*}
Now observe that by \eqref{Goku_alpha} (with $j=1$) we get
\begin{align}
\|a_{1,\gamma}(t,x+h)-a_{1,\gamma}(t,x)\|_{\mathcal{L}(E;\R)} &\leq 2c_\alpha \pa{\int_0^{\min\{t,\|h\|_E^{2/(1+\gamma)}\}}s^{-(1-\alpha)(1+\gamma)/2}ds}\|g\|_{{\rm BUC}^{0,\alpha}([0,T]\times E)}\notag\\
& \leq \frac{4 c_\alpha}{2-(1-\alpha)(1+\gamma)}\|h\|_E^{\alpha+(1-\gamma)/(1+\gamma)}\|g\|_{{\rm BUC}^{0,\alpha}([0,T]\times E)}.\label{HULK_HOGAN1}
\end{align}
Observe that $b_{1,\gamma}(t,x+h)-b_{1,\gamma}(t,x)=0$ if $\|h\|_E^{2/(1+\gamma)}\geq t$. So if $\|h\|_E^{2/(1+\gamma)}<t$, by \eqref{Goku_alpha} (with $j=2$), we get 
\begin{align}
&\|b_{1,\gamma}(t,x+h)-b_{1,\gamma}(t,x)\|_{\mathcal{L}(E;\R)}\notag\\
&=\norm{\int_{\min\{t,\|h\|_E^{2/(1+\gamma)}\}}^t\int_0^1\D^2P(s)g(t-s,\cdot)(x+\sigma h)(h,\cdot) d\sigma ds}_{\mathcal{L}(E;\R)}\notag\\
&\leq c_\alpha\|h\|_E\pa{\int_{\|h\|_E^{2/(1+\gamma)}}^1s^{-(2-\alpha)(1+\gamma)/2}ds+\chi_{(1,+\infty)}(T)\int_1^Tds}\|g\|_{{\rm BUC}^{0,\alpha}([0,T]\times E)}\notag\\
&\leq c_\alpha\|h\|_E\pa{\int_{\|h\|_E^{2/(1+\gamma)}}^{+\infty}s^{-(2-\alpha)(1+\gamma)/2}ds+\chi_{(1,+\infty)}(T)\int_1^T\frac{1}{\|h\|_E^{2\gamma/(1+\gamma)-\alpha}}ds}\|g\|_{{\rm BUC}^{0,\alpha}([0,T]\times E)}\notag\\
&= c_\alpha\pa{\frac{2}{(2-\alpha)(1+\gamma)-2}+(T-1)\chi_{(1,+\infty)}(T)}\|h\|_E^{\alpha+(1-\gamma)/(1+\gamma)}\|g\|_{{\rm BUC}^{0,\alpha}([0,T]\times E)}.\label{HULK_HOGAN2}
\end{align}
Combining \eqref{HULK_HOGAN1} and \eqref{HULK_HOGAN2} we get \eqref{Thm_evol_Schauder_Rd1}, indeed the case $\|h\|_E\geq 1$ can be obtained using arguments similar to those used in the proof of \eqref{caso>1,2}.
\end{proof}

\section{Remarks and Examples}\label{RMKS}
In this section we present some relevant examples of $A$, $b$, $\OO$ and $\gamma\in [0,1]$  verifying Hypotheses \ref{RDS1}. The following is a simple example of function $b$ that verify Hypothesis \ref{RDS1}\eqref{bbbb}
\[
b(\xi,x):=-C_{2m+1}(\xi)x^{2m+1}+\sum^{2m}_{k=0}C_k(\xi)x^k, \qquad \xi\in \CO,\ x\in\R;
\]
where $C_0,\ldots,C_{2m}$ are continuous functions from $\CO$ to $\R$ and $C_{2m+1}:\CO\ra\R$ is continuous and strictly positive function. 
If $d=1$ and $\mathcal{O}$ is any compact interval, then Hypothesis \ref{RDS1}\eqref{gamma} is verified for any $\gamma\in[0,1]$. While if $d=2,3$ then Hypothesis \ref{RDS1}\eqref{gamma} may not be verified for some $\gamma\in [0,1]$. However, if $\OO$ is ``regular'' enough, then there exists $\gamma_\OO>0$ such that Hypothesis \ref{RDS1}\eqref{gamma} holds true for any $\gamma>\gamma_\OO$. The constant $\gamma_\OO$ depends on the dimension $d$ and on the boundary of $\OO$. We refer to \cite[Section 4.2]{FU-OR1} for a proof of the following proposition.
\begin{prop}\label{Carlo}
Let $d=1,2,3$. If $A$ is a realization of the Laplacian operator with Dirichlet boundary condition, then Hypothesis \ref{RDS1}\eqref{gamma} holds true in the following cases
\begin{align*}
\gamma &> \frac{2d-3}{2},\qquad\mathcal{O}:=\{x\in \R^d\, |\, \norm{x}_{\R^d}\leq 1 \};\\
\gamma &> \frac{d-2}{2},\qquad\mathcal{O}:=[0,\pi]^d.
\end{align*}
\end{prop}
In view of Proposition \ref{Carlo}, we can deduce some simple examples where the main results of this papers can be applied. However by Proposition \ref{Carlo} the results of this paper cannot be applied to the case $\mathcal{O}=\{x\in \R^3\; \norm{x}_{\R^3}\leq 1\}$. Indeed, in this case, Hypothesis \ref{RDS1}\eqref{gamma} is verified if $\gamma>3/2$, however we do not know of a technique to obtain estimates similar to \eqref{Goku} and \eqref{Gokulip} in the case $\gamma>1$.

\begin{rmk}
If $A:\Dom(A)\subseteq C(\ol{\mathcal{O}})\ra C(\ol{\mathcal{O}})$ is the realization of the Laplacian operator with Neumann boundary conditions then $E=C(\ol{\mathcal{O}})$.
\end{rmk}

Even in the linear case ($F\equiv0$ in \eqref{eqF02}) it does not seem reasonable to obtain a maximal Schauder regularity result if $\gamma>0$ in \eqref{eqF02} (see, for example, \cite{LR21}). Here with maximal Schauder regularity result we mean that if $f\in C_b^\alpha(E)$, then $u\in C^{2+\alpha}_b(E)$. In some recent papers, in order to recover the missing regularity and obtain results similar to the finite dimensional one, the authors work along ``privileged directions'' (see \cite{ABF21,BF20,BF22,BF_Schauder,CL19,LR21,MAS1}). However, to prove a result similar to that in \cite{BF_Schauder,CL19,LR21} for \eqref{eqF02} we shall need that the function $b$ verifies a dissipativity assumption similar to Hypothesis \ref{RDS1}\eqref{bbbb} but along the directions of $(-A)^{-\gamma/2}(E)$, which, right now, we are not able to prove.

\appendix

\section{A result about uniformly continuous functions}

We recall a result reguarding uniformly continuous functions that we have used throughout the paper. The proof is standard and follows the same ideas of the proof of \cite[Lemma 3.3]{Pri99}, we provide it for completeness (see also \cite{BF_Schauder}).

\begin{prop}\label{Zone400}
Let $\mathcal{M}$ be a separable metric space with metric $d_{\mathcal{M}}$, $(Y,\mu)$ be a measurable space ($\mu$ is a finite, positive and complete measure) and $Z$ be a Banach space with norm $\norm{\cdot}_Z$. Consider a function $F:\mathcal{M}\times Y\ra Z$ that satisfies
\begin{enumerate}[\rm (i)]
\item for any $m\in\mathcal{M}$, the map $y\mapsto F(m,y)$ is measurable;

\item for $\mu$-a.e. $y\in Y$, the map $m\mapsto F(m,y)$ is uniformly continuous;\label{Zone400,2}

\item there exists a $\mu$-integrable function $g:Y\ra\R$ such that for all $m\in\mathcal{M}$ and $\mu$-a.e. $y\in Y$ it holds $\|F(m,y)\|_Z\leq g(y)$.
\end{enumerate}
The map $h:\mathcal{M}\ra Z$, defined as
\begin{align*}
h(m):=\int_Y F(m,y) \mu(dy),\qquad m\in\mathcal{M},
\end{align*}
is bounded and uniformly continuous.
\end{prop}

\begin{proof}
The boundedness of $h$ is trivial and its continuity follows by the dominated convergence theorem (see \cite[Theorem 3, p.45]{DU77}). So we just need to prove the uniform continuity of $h$. Let $N\subseteq Y$ be such that $\mu(N)=0$ and condition \eqref{Zone400,2} holds for all $y\in Y\setminus N$. For any $y\in Y\setminus N$, consider the modulus of continuity of the map $x\mapsto F(x,y)$ defined as
\begin{align*}
\omega_{F,y}(t):=\sup\set{\|F(m,y)-F(m',y)\|_Z\tc m,m'\in\mathcal{M},\ d_{\mathcal{M}}(m,m')\leq t},\qquad t>0.
\end{align*}
Observe that, by the separability of $\mathcal{M}\times\mathcal{M}$, for every $t>0$ there exists a countable set $D(t)\subseteq \mathcal{M}\times\mathcal{M}$ such that 
\begin{align*}
\omega_{F,y}(t):=\sup\set{\|F(m,y)-F(m',y)\|_Z\tc (m,m')\in D(t),\ d_{\mathcal{M}}(m,m')\leq t},\qquad t>0.
\end{align*}
The countability of $D(t)$ assures the measurability, with respect to $\mu$, of the map $y\mapsto\omega_{F,y}(t)$ for any $y>0$.
Moreover a standard computation gives that for any $y\in Y\setminus N$ and $t>0$ it holds $\omega_{F,y}(t)\leq 2g(y)$. Now let $(t_n)_{n\in\N}$ be a sequence of positive real numbers converging to zero. We have, for any $m,m'\in\X$ with $d_{\mathcal{M}}(m,m')\leq t_n$ 
\begin{align*}
\|h(m)-h(m')\|_Z\leq \int_Y\|F(m,y)-F(m',y)\|_Z\mu(dy)\leq \int_Y\omega_{F,y}(t_n)\mu(y).
\end{align*}
The thesis follows by the dominated convergence theorem (see \cite[Theorem 3, p.45]{DU77}).
\end{proof}

\section*{Declarations}


\subsection*{Fundings} The authors are members of GNAMPA (Gruppo Nazionale per l’Analisi Matematica, la Probabilit\`a
e le loro Applicazioni) of the Italian Istituto Nazionale di Alta Matematica (INdAM). The authors have been also partially supported by the research project PRIN 2015233N5A ``Deterministic and stochastic evolution equations'' of the Italian Ministry of Education, MIUR. S.F has been partially supported by the INdAM-GNAMPA project ``Problemi ellittici e pa\-ra\-bo\-li\-ci in dimensione infinita'' CUP\_E55F22000270001. D.A.B. has been partially supported by the grant ``Stochastic differential equations and associated Markovian semigroups'' of the University of Pavia and by the INdAM-GNAMPA project ``Analisi qualitativa di PDE stocastiche: ergodicit\`a ed equazioni di Kolmogorov'' CUP\_E55F22000270001. The authors have no relevant financial or non-financial interests to disclose.

\subsection*{Research Data Policy and Data Availability Statements} Data sharing not applicable to this article as no datasets were generated or analysed during the current study.

\end{document}